\pdfpkmode={lexmarkr}
\pdfmapfile{}

\documentclass[a4paper,reqno,11pt]{amsart}
\usepackage{amsmath}
\usepackage{amssymb}
\usepackage{amsthm}
\usepackage{graphicx}
\usepackage{textcomp}
\usepackage{mathtools}
\usepackage{enumerate}
\usepackage{xcolor}
\usepackage{supertabular}
\usepackage{booktabs}
\usepackage{csquotes}
\usepackage{url}
\usepackage[section]{placeins}
\usepackage[margin=3.5cm]{geometry}

\newtheorem{thm}{Theorem}[section]
\newtheorem{prop}[thm]{Proposition}
\newtheorem{lemma}[thm]{Lemma}

\theoremstyle{remark}

\newtheorem{remark}[thm]{Remark}
\theoremstyle{definition}

\numberwithin{equation}{section}
\numberwithin{table}{section}
\numberwithin{figure}{section}

\DeclareMathOperator{\Int}{Int}

\let\phi=\varphi
\let\epsilon=\varepsilon

\title{On nodal domains in Euclidean balls}

\author{Bernard Helffer}
\author{Mikael Persson Sundqvist}
\address[Bernard Helffer]{ Laboratoire de
Math\'{e}matiques UMR CNRS 8628\\ Universit\'{e} Paris-Sud - B\^{a}t 425\\
F-91405 Orsay Cedex\\ France  and Laboratoire de Math\'ematiques Jean Leray, 
Universit\'e de Nantes, France.}
\email{bernard.helffer@math.u-psud.fr}
\address[Mikael Persson Sundqvist]{Lund University, Department of Mathematical 
Sciences, Box 118, 221\ 00 Lund, Sweden.}
\email{mickep@maths.lth.se}

\begin{document}
\begin{abstract}
\AA. Pleijel (1956) has proved that in the case of the Laplacian  with 
Dirichlet condition, the equality in the Courant nodal theorem (Courant 
sharp situation) can only be true for a finite number of eigenvalues when the 
dimension is $\geq 2$. Recently Polterovich extended the result
to the Neumann problem in two dimensions in the case when the boundary is 
piecewise analytic.
A question coming from the theory of spectral minimal partitions has 
motivated the analysis of the cases when one has equality in Courant's 
theorem. 

We identify the Courant sharp eigenvalues for the Dirichlet and the
Neumann Laplacians in balls in $\mathbb R^d$, $d\geq 2$. It is the first 
result of this type holding in any dimension.
The corresponding result for the Dirichlet Laplacian in the disc in $\mathbb R^2$
was obtained by B.~Helffer, T.~Hoffmann-Ostenhof and 
S.~Terracini.
\end{abstract}

\maketitle

\section{Introduction and main results}
We consider the problem of counting nodal domains of eigenfunctions of the 
self-adjoint realization of the Laplacian, $-\Delta$ in the unit ball
in $\mathbb R^d$. The \enquote{nodal domains}  are the connected components of 
the zeroset of the eigenfunction in the ball.  We consider the Dirichlet 
problem for $d\geq 3$ and
the Neumann problem for $d\geq 2$ (the corresponding results for the
Dirichlet problem for $d=2$ was given in~\cite{HHOT}).

To be more precise, denoting by $\lambda_n$ the $n$th eigenvalue, our goal 
is to discuss the property of Courant sharpness of these operators, that
is the existence of eigenvalues $\lambda_n$ for which there exists an
eigenfunction with exactly $n$ nodal domains.  We recall that Courant's theorem
says that the number of nodal domains, $\mu(\Psi)$, of an eigenfunction $\Psi$
corresponding to $\lambda_n$ is bounded by~$n$. Moreover, it has been proven 
that the number of Courant sharp cases must be finite, see~\cite{Pl} for 
the Dirichlet case and~\cite{Pol} for the Neumann case (in dimension $2$ only 
and for piecewise analytic boundaries). The two first eigenvalues 
are always Courant sharp. We will prove the following.
\begin{thm}
\label{thm:2D}
The only Courant sharp eigenvalues for the Neumann Laplacian for the disc
are $\lambda_1$, $\lambda_2$ and $\lambda_4$.
\end{thm}

\begin{thm}
\label{thm:dD}
The only Courant sharp eigenvalues for the Dirichlet and Neumann Laplacians
for the ball in $\mathbb R^d$, $d\geq 3$, are $\lambda_1$ and $\lambda_2$.
\end{thm}

This analysis is motivated by the problem of spectral minimal $k$-partitions, 
where one is interested in minimizing $\max_j \lambda_1(D_j)$ over the 
family $\mathcal D =(D_1,\cdots, D_k)$ of pairwise disjoint open sets in 
a domain $\Omega$, where $\lambda_1(D_j)$ denotes either the 
Dirichlet ground 
state energy (if we analyze the Dirichlet spectral partitions of an open set 
$\Omega$) or the Dirichlet--Neumann ground state energy for the Laplacian in 
$D_j$ with  Neumann condition on 
$\partial D_j \cap \partial \Omega$ and Dirichlet condition on 
the remaining part of $\partial D_j$. There are now many results in the 
two-dimensional (2D) case. We refer to~\cite{BH} for a recent review. 
In higher dimensions much less is done, and we only know of the 
determination of all Courant sharp Dirichlet eigenvalues 
in the cube in three dimensions,~\cite{HK}. We also know less 
about the properties of $k$-minimal partitions in higher dimensions. 
This will not create too much problems below, because we will work with 
explicit nodal domains of eigenfunctions,  which in spherical coordinates 
will be expressed as a product of an interval (in the radial direction) 
by a nodal domain of a spherical harmonics in $\mathbb S^{d-1}$.

In Section~\ref{sec:spectrum} we recall how one  describes the spectrum of the 
Laplace operator.
As a part of the analysis of the Neumann problem, we use and extend a 
recent result on the zeros of derivatives of the Bessel functions 
$J_\nu$, saying that $J_\nu'$ and $J_{\nu+p}'$ have no common positive zeros 
if $\nu\geq 0$ and $p\geq 1$ are integers. This was
proved by M.~Ashu in his Bachelor thesis~\cite{As}.

In Section~\ref{sec:Courant} we discuss Courant sharpness. As a first result,
we use a symmetry argument to extend a result by Leydold (\cite{Ley0,Ley1}) 
from $\mathbb S^2$ to $\mathbb S^{d-1}$, $d\geq 4$,
saying that only the first two eigenvalues of the Laplace--Beltrami operator
on $\mathbb S^{d-1}$ are Courant sharp. We then continue towards the 
proofs of the Theorems~\ref{thm:2D} and~\ref{thm:dD}, by reducing the number
of cases that need special treatment by using what we call a twisting trick. In
short, it says that if the eigenfunction is non-radial, and if the eigenfunction
is zero on a set $|x|=\rho$, $\rho<1$, then one can consider the same
eigenfunction, but where one makes a small rotation of the inner ball $|x|<\rho$, 
breaking the necessary symmetry. We refer to Subsection~\ref{sec:Twist} for the
full details. 
This leaves two families of eigenfunctions to consider. In 
Subsection~\ref{sec:Dirichlet} we finish the proof of Theorem~\ref{thm:dD} in
the case of Dirichlet boundary condition, by using an interlacing property
of zeros of Bessel functions. In Subsection~\ref{sec:Neumann} we finish the
proof of Theorem~\ref{thm:2D} and the Neumann part of Theorem~\ref{thm:dD}.
We remark that the proof of Theorem~\ref{thm:2D} is quite close to the proof 
of the Dirichlet case for the disc~\cite[Section~9]{HHOT}.

In Section~\ref{sec:Pleijel}, we discuss the possible extension of a
theorem by \AA.~Pleijel,~\cite{Pl}. The question is to determine if there exists 
a constant $\gamma<1$ such that, for any infinite sequence of 
eigenpairs $(\lambda_n,u_n)$
\[ 
\lim\sup \frac{\mu(u_n)}{n} \leq \gamma\,.
\]
For the Dirichlet problem, this is indeed the case as proved in the paper of 
B\'erard-Meyer~\cite{BeMe}, which establishes, in any dimension $d\geq 2$ 
for bounded open sets in $\mathbb R^d$ or $d$ dimensional compact 
Riemannian manifolds, the existence of an explicit universal 
constant $\gamma(d)<1$ (extending~\cite{Pe}). This was also solved previously 
for the Neumann problem in dimension 2~\cite{Pol}.

Finally, in Section~\ref{sec:gamma}, we establish new monotonicity 
properties of the function $\gamma(d)$.

\begin{remark}
It would be interesting to consider the problem of minimal $k$-partitions of 
the ball in three dimensions. 
In the case $k=3$, it has been proved in~\cite{HHOT2} that the 
minimal $3$-partition of the sphere $\mathbb S^2$ is up to rotation determined 
by the intersection of $\mathbb S^2$ with three half-planes crossing along the 
vertical axis with 
equal angle $\frac {2\pi}{3}$. It is natural to conjecture that the minimal 
$3$-partition for the ball is up to rotation determined by the intersection 
of the ball with three half-planes crossing along the vertical axis with 
equal angle $\frac {2\pi}{3}$. 
\end{remark}

\section{Spectrum of the Laplace operator in the unit ball in $\mathbb R^d$}
\label{sec:spectrum}

We denote by $-\Delta^D$ and $-\Delta^N$ the Dirichlet and Neumann Laplace
operators, respectively, in the unit ball in $\mathbb R^d$, $d\geq 2$.
The Laplace operator $-\Delta$ can be written as
\[
-\Delta=-\frac{\partial^2}{\partial r^2}-\frac{d-1}{r}\frac{\partial}{\partial r}
+\frac{1}{r^2}(-\Delta_{\mathbb S^{d-1}}),
\]
where $r=|x|$ is the radial variable and $\Delta_{\mathbb S^{d-1}}$ is the
Laplace--Beltrami operator, acting in $L^2(\mathbb S^{d-1})$.

\begin{prop}[{\cite[Theorem~22.1 and Corollary~22.1]{Shubin}}]\label{proshu} 
Assume that $d\geq 2$. 
The spectrum of $-\Delta_{\mathbb S^{d-1}}$ consists of eigenvalues
\[
\ell(\ell+d-2),\quad \ell\in\mathbb{N}\,.
\]
The multiplicity of the eigenvalue $\ell(\ell+d-2)$ is given by
\[
\Lambda_{\ell,d}:=\binom{\ell+d-1}{d-1}-\binom{\ell+d-3}{d-1}\,,
\]
which coincides with the dimension of the space of homogeneous, 
harmonic polynomials of degree $\ell$.
\end{prop}

This leads us to consider the Dirichlet and Neumann eigenvalues of the 
ordinary differential operator
\[
\mathcal L=-\frac{d^2}{dr^2}-\frac{d-1}{r}\frac{d}{dr}+\frac{\ell(\ell+d-2)}{r^2}\,,
\]
acting in $L^2((0,1),r^{d-1}\,dr)\,$.\\
 The general solution to 
$\mathcal Lu=\lambda u$ is given by
\[
u(r)=c_1r^{\frac{2-d}{2}}J_{\frac{1}{2}(2\ell+d-2)}(\sqrt{\lambda}r)
+c_2r^{\frac{2-d}{2}}Y_{\frac{1}{2}(2\ell+d-2)}(\sqrt{\lambda}r)\,,
\]
where $J_\nu$ and $Y_\nu$ denote the Bessel functions of order $\nu$, and of 
first and second kind, respectively. The Bessel functions of the second kind
are too singular at the origin to be considered as eigenfunctions.

To state the next results, we introduce the function
\[
\Xi_{\ell}^{(d)}(r)=r^{\frac{2-d}{2}}J_{\frac{1}{2}(2\ell+d-2)}(r)\,,
\]
which is also denoted $\Xi_{\ell}$ for simplicity.

\begin{prop}
\label{prop:specD}
The spectrum of $-\Delta^D$ in the unit ball in $\mathbb R^d$, $d\geq 2$\,,
consists of eigenvalues
\[
\lambda_{\ell,m}^D=\bigl(\alpha_{\ell,m}^{(d)}\bigr)^2\,,
\quad \ell\in\mathbb{N}\,,\ m\in\mathbb{N}\setminus\{0\},
\]
where $\alpha_{\ell,m}^{(d)}$ denotes the $m$th positive zero of
the function $\Xi_{\ell}^{(d)}$. Each eigenvalue has 
multiplicity $\Lambda_{\ell,d}\,$.
\end{prop}

\begin{prop}
\label{prop:specN}
The spectrum of $-\Delta^N$ in the unit ball in $\mathbb R^d$, $d\geq 2$\,,
consists of eigenvalues
\[
\lambda_{\ell,m}^N=\bigl(\beta_{\ell,m}^{(d)}\bigr)^2\,,
\quad  \ell\in\mathbb{N}\,,\ m\in\mathbb{N}\setminus\{0\}\,,
\]
where $\beta_{\ell,m}^{(d)}$ denotes the $m$th positive (non-negative if 
$\ell=0$) zero of the function 
$r\mapsto\frac{d}{dr}\Xi_{\ell}^{(d)}(r)$. 
Each eigenvalue has multiplicity $\Lambda_{\ell,d}\,$.
\end{prop}

The only statement in these propositions that needs a proof is that of the 
multiplicity of the eigenvalues. For the Dirichlet case the needed result 
is given in~\cite[\S 15.28]{Wa}. It says  that the Bessel functions $J_\nu$
and $J_{\nu+p}$ do not have any common positive zeros. This was conjectured
by Bourget (1866), and follows from  a deep result obtained by  
Siegel~\cite{Sie} in 1929. He proved that if $r>0$ is an algebraic number, 
and $\nu\in\mathbb Q$, then $J_\nu(r)$ is not an algebraic number.

The corresponding result for the Neumann problem was solved recently in the
case $d=2$ in Ashu's Bachelor thesis,~\cite{As}.
In this particular case the statement is that $J_\nu'$ and $J_{\nu+p}'$ have 
no common positive zeros. Again, there is a deep result behind, given 
in~\cite[page~217]{Sh}, which we will come back to in the proof of the 
first lemma below.

\begin{lemma}
\label{lem:Xitrans}
Assume that $d\geq 2$ and that $\ell\in\mathbb N $. Then the positive
zeros of the function $\Xi_{\ell}^{(d)}$ are transcendental numbers.
\end{lemma}

\begin{proof}
The functions $K_\nu$  (not to be mixed up with the modified Bessel
functions) are introduced in~\cite{Sh} via the identity
\[
J_\nu(r)=\frac{1}{\Gamma(\nu+1)}\Bigl(\frac{r}{2}\Bigr)^\nu K_\nu(r)\,.
\]
We express the derivative of $\Xi_{\ell}^{(d)}$ in terms of these $K$ functions,
\begin{equation}
\label{eq:Xiellprime}
\frac{d}{dr}\Xi_\ell^{(d)}(r)=r^{-d/2}\frac{1}{\Gamma(\ell+d/2)}
\Bigl(\frac{r}{2}\Bigr)^{\ell+d/2-1}
\biggl[
\ell K_{\ell+d/2-1}(r)+r K_{\ell+d/2-1}'(r)
\biggr]\,.
\end{equation}
Assume that $r>0$ is an algebraic zero of 
$r\mapsto\frac{d}{dr}\Xi_{\ell}^{(d)}(r)$. 
Then both $K_{\ell+d/2-1}(r)$ and $K_{\ell+d/2-1}'(r)$ 
are transcendental according to~\cite[Theorem~6.3]{Sh}. In particular they
are non-zero. However, as 
noted in~\cite[page~217]{Sh}, also $K_{\ell+d/2-1}'(r)/K_{\ell+d/2-1}(r)$ is
transcendental. But then $\ell/r$ is transcendental by~\eqref{eq:Xiellprime}. 
Since $\ell$ is an integer and $r$ was assumed to be algebraic, 
this is a contradiction.
\end{proof}

Proposition~\ref{prop:specN} is a direct consequence of this lemma.

\begin{lemma}
\label{lem:Xizeros}
Assume that $d\geq 2$, $\ell\in\mathbb N$ and  $p \in \mathbb N \setminus \{0\}$. 
Then the functions $r\mapsto \frac{d}{dr}\Xi_{\ell}^{(d)}$ and 
$r\mapsto \frac{d}{dr}\Xi_{\ell+p}^{(d)}$ have
no common positive zeros.
\end{lemma}
Before giving the proof, we recall some recursion formulas for the
Bessel functions, valid for all $\nu\in\mathbb R$ and positive $r$,
\begin{align}
\label{eq:Jp1}
J_\nu'(r)&=\frac{\nu}{r} J_\nu(r)-J_{\nu+1}(r)\,,\\
\label{eq:Jp2}
J_{\nu}'(r)&=-\frac{\nu}{r} J_{\nu}(r)+ J_{\nu-1}(r)\,,\\
\label{eq:Jrec}
J_{\nu+1}(r)&=\frac{2\nu}{r}J_{\nu}(r)-J_{\nu-1}(r)\,.
\end{align}

\begin{proof}[{Proof of Lemma~\ref{lem:Xizeros}}]
By~\eqref{eq:Jp1}--\eqref{eq:Jrec}, we get the corresponding
formulas for $\Xi_{\ell}\,$,
\begin{align}
\label{eq:Xip1}\Xi_{\ell}'(r)&=\frac{\ell}{r} \Xi_{\ell}(r)-\Xi_{\ell+1}(r)\,,
\quad \ell\geq 0\,,\\
\label{eq:Xip2}\Xi_{\ell}'(r)&=-\frac{\ell+d-2}{r}\Xi_{\ell}(r)+\Xi_{\ell-1}(r)\,,
\quad \ell\geq 1\,,\\
\label{eq:Xirec} \Xi_{\ell}(r)&=\frac{2\ell+d-4}{r}\Xi_{\ell-1}(r)-\Xi_{\ell-2}(r)\,,
\quad \ell\geq 2\,.
\end{align}

We divide the proof into different cases, and do the proof by contradiction,
using recursion formulas and Lemma~\ref{lem:Xitrans}.

\emph{Case 1, $\ell=0$ and $p=1\,$}:

If, for $r>0$, $\Xi_0'(r)=\Xi_1'(r)=0\,$, 
then~\eqref{eq:Xip1} with $\ell=0$ implies that $\Xi_{1}(r)=0\,$, 
which contradicts Cauchy uniqueness.

\emph{Case 2, $\ell=0$ and $p\geq 2\,$}:

Assume that $r>0$ is a zero of $\Xi_{0}'$ and $\Xi_{p}'\,$. As in Case~1, 
we find that $\Xi_{1}(r)=0$, and so by~\eqref{eq:Xirec}, $\Xi_{2}(r)=-\Xi_0(r)\,$. 
One application of~\eqref{eq:Xip2} gives
\[
0=\Xi_{p}'(r)=-\frac{p+d-2}{r}\, \Xi_p(r)+\Xi_{p-1}(r)\, .
\]
Next, we use~\eqref{eq:Xirec} several times to reduce the right-hand side to
an expression involving $\Xi_2(r)$ and $\Xi_1(r)$ only. After $p-2$ applications
we find a polynomial $Q$ in the variable $1/r$ times $\Xi_2(r)$ only, since 
$\Xi_1(r)=0\,$. The highest degree term of the polynomial is
\[
-\frac{p+d-2}{r}\, \frac{2p+d-4}{r}\, \frac{2p+d-6}{r}\cdots \frac{2p+d-(2p-2)}{r}\,.
\]
Since $\Xi_2(r)=-\Xi_0(r)\,$, we find that
\[
0= Q(1/r)\Xi_0(r)\,,
\]
where $Q$ is a non-vanishing polynomial with rational coefficients. 
Since $r$ is transcendental by Lemma~\ref{lem:Xitrans}, $Q(1/r)\neq 0\,$. 
But $\Xi_0(r)\neq 0$ by Cauchy uniqueness, so we end up at a contradiction and
conclude that $\Xi_0'$ and $\Xi_p'$ have no common positive zero.

\emph{Case 3, $\ell\geq 1$ and $p\geq 1\,$}:

Again, assume that $r>0$ is a zero of $\Xi_{\ell}'$ and $\Xi_{\ell+p}'\,$.
This means, using~\eqref{eq:Xip1} and~\eqref{eq:Xip2} respectively,
\begin{align}
\label{eq:Xihelp}
0&=\Xi_{\ell}'(r)=-\Xi_{\ell+1}(r)+\frac{\ell}{r}\Xi_{\ell}(r)\,,\\
\nonumber
0&=\Xi_{\ell+p}'(r)=-\frac{\ell+p+d-2}{r}\Xi_{\ell+p}(r)+\Xi_{\ell+p-1}(r)\,.
\end{align}
We use~\eqref{eq:Xirec} repeatedly, to reduce the second equation
so that it involves only $\Xi_{\ell}(r)$ and $\Xi_{\ell+1}(r)$, with polynomial
(in the variable $1/r$) coefficients in front. The highest degree (in $1/r$) 
coefficient in front of $\Xi_{\ell+1}(r)$ will, after $p-1$ steps, become
\[
-\frac{\ell+p+d-2}{r}\frac{2\ell+2p+d-4}{r}\frac{2\ell+2p+d-6}{r}
\cdots\frac{2\ell+2p+d-2p}{r}\,,
\]
and once reduced, while calculating the determinant of the resulting system, 
this term will be multiplied with $\ell/r$ (that is in front of 
$\Xi_{\ell}(r)$ in~\eqref{eq:Xihelp}), 
which will higher its degree (in $1/r$) by one. No such term can occur 
elsewhere, and thus for the determinant of the system to be zero, $r$ must 
solve a polynomial equation with rational coefficients, so $r$ is 
algebraic. That contradicts Lemma~\ref{lem:Xitrans}. The other possibility is 
that $\Xi_{\ell}(r)=\Xi_{\ell+1}(r)=0\,$. But that would imply that 
$\Xi_{\ell}(r)=\Xi_{\ell}'(r)=0\,$, which, again, contradicts the 
Cauchy uniqueness.
\end{proof}

\section{Courant sharpness}
\label{sec:Courant}
\subsection{The result on $\mathbb S^{d-1}$}
We first analyze the case of the sphere and extends Leydold's result to 
$\mathbb S^{d-1}$ for $d\geq 3$.
\begin{thm}
\label{thm:CourantSphere}
If $d\geq 3$, the only Courant sharp cases for the Laplace--Beltrami operator on 
$\mathbb S^{d-1}$ correspond to the two first eigenvalues.
\end{thm}

In the proof we need the following version of Courant's theorem with symmetry 
(see for example~\cite[Subsection~2.4]{BH}) which we also prove for the
sake of completeness.

\begin{thm}
\label{thm:Courantsymmetry}
Given an eigenfunction which is symmetric or antisymmetric with respect to the 
antipodal map, the number of its nodal domains is
not greater than two times the smallest labeling of the
corresponding eigenvalue inside its symmetry space.
\end{thm}

\begin{proof}
We note
that each eigenspace has a specific symmetry with respect to the antipodal map. 
An eigenfunction $\psi_\ell$ associated with the eigenvalue $\ell (\ell + d-2)$ 
satisfies indeed
\[
\psi_\ell (-\omega) = (-1)^\ell \psi_\ell  (\omega)\,,
\quad\forall \omega \in \mathbb S^{d-1}\,.
\]
This is an immediate consequence of the fact that $\psi_\ell$ is the restriction 
to $\mathbb S^{d-1}$ of an homogeneous polynomial of degree $\ell$ of 
$d$ variables.

With this in mind, we first assume that $\ell$ is odd, and hence let 
$\psi_\ell$ be an eigenfunction with minimal labeling $\nu$ inside the 
antisymmetric space. Let us assume, to get a contradiction, that 
\[
\mu(\psi_\ell)  \geq 2 \nu +1\,.
\]
We note that by antisymmetry, $\mu(\psi_\ell)$ is even. Hence we would have 
actually
\[
\mu(\psi_\ell)  \geq 2 \nu +2\,.
\]
We now follow the standard proof of Courant's theorem.
Selecting $(\mu(\psi_\ell)/2\, -1)$ pairs of symmetric nodal domains, we can 
construct an antisymmetric function, which is orthogonal to the antisymmetric 
eigenspace corresponding to the $\nu -1$ first eigenvalues and has an energy not 
greater than the $\nu$-th eigenvalue. Using the mini-max characterization of 
the $\nu$:th eigenvalue, we get that this function is an antisymmetric 
eigenfunction which vanishes in the two remaining nodal domains. This gives the 
contradiction using the unique continuation principle.

Next, assume that $\ell$ is even and that $\psi_\ell$ is an eigenfunction 
with minimal labeling~$\nu$ inside the symmetric space. We assume, again to
get a contradiction, that 
\[
\mu(\psi_\ell)  \geq 2 \nu +1\,.
\]
We have 
\[
\mu(\psi_\ell)  = \mu' + 2 \mu''
\]
where $\mu'$ is the number of nodal domains which are symmetric and $\mu''$ 
is the number of pairs of nodal domains which are exchanged by symmetry.

If $\mu' =0$, the proof is identical to the antisymmetric one. If 
$\mu' \geq 1$,  we can select $\mu'-1$ symmetric nodal domains and $\mu''$ 
pairs of nodal domains exchanged by symmetry and construct  a symmetric 
function which is orthogonal to the symmetric eigenspace corresponding to 
the $\nu -1$ first eigenvalues and has an energy not greater than the 
$\nu$:th eigenvalue. Here we have used our assumption by contradiction to 
get that  $\mu'-1 + 2\mu'' \geq \nu$. We get a contradiction just as 
before.
\end{proof}

\begin{proof}[Proof of Theorem~\ref{thm:CourantSphere}]
This is just an adaptation of Leydold's proof (\cite{Ley0,Ley1}). 

We consider the (smallest) labeling of the eigenvalue $\ell (\ell +d-2)$, i.e. the smallest $n$ such that $\lambda_n=\ell(\ell+d-2)$.
According to Proposition~\ref{proshu}, the smallest labeling of 
the eigenvalue $\ell(\ell +d-2)$ is obtained by~$1$ if $\ell=0$,~$2$ if 
$\ell=1$, and
\[
1 + \binom{\ell+d -2}{d-1} + \binom{\ell +d -3}{d-1}\,,\quad\forall\ell\geq 2\,.
\]
Using Theorem~\ref{thm:Courantsymmetry} on the eigenvalue $\ell (\ell+d-2)$, 
we get
\[
\mu(u_\ell) \leq 2 \biggl[ \binom{\ell + d-3}{d-1} +1 \biggr]\,.
\]
To compute this labeling we have used that for a given $\ell$, the labeling 
is obtained by adding $1$ to the sum of the multiplicity associated with 
the $\ell' < \ell$ with the same parity as $\ell$.

Hence, we have to check that if $\ell \geq 2$ and $d\geq 3\,$, then
\begin{equation}
\label{eq:binomials}
2 \biggl[\binom{\ell +d -3}{d-1} +1\biggr] 
< 1 + \binom{\ell +d-2 }{d-1} + \binom{\ell+d -3}{d-1}\,.
\end{equation}
Since $\binom{\ell+d-2}{d-1}=\binom{\ell+d-3}{d-2}+\binom{\ell+d-3}{d-1}\,$,
the inequality~\eqref{eq:binomials} reads $1 < \binom{\ell + d-3}{d-2}\,$, which 
is satisfied when $\ell \geq 2$ and $d \geq 3\,$.
\end{proof}

\subsection{Twisting trick}
\label{sec:Twist}

\begin{lemma}
\label{lem:twist}
If $\ell\geq 1$ and $m\geq 2$ then neither $\lambda_{\ell,m}^D$ nor
$\lambda_{\ell,m}^N$ can be Courant sharp.
\end{lemma}

Because the theory of 
minimal partitions has not been developed to the same extend when $d>3$, we 
explain how the proof goes, without referring to~\cite{HHOT, HHOT1} which are 
mainly devoted to the case when the dimension is 2 or 3. 
The proof below is somewhat reminiscent of a proof written in 
collaboration with T.~Hoffmann-Ostenhof (2005), which was never published but 
is mentioned in~\cite{HHOT}.

\begin{proof}
We start with the Dirichlet situation, and omit the $D$ in the notation. All
eigenvalues occurring are Dirichlet eigenvalues of the Laplace operator. The
domain will differ, and we will be explicit about that.

Assume that we have a Courant sharp eigenvalue $\lambda_n=\lambda_{\ell,m}\,$,
with $\ell\geq 1$ and $m\geq 2$. We will construct a partition 
$\widehat{\mathcal D}$ of $n$ non-intersecting open sets 
$\{\widehat{D}_j\}_{j=1}^n$ in the ball, such that
\[
\max_j\lambda_1(\widehat D_j)<\lambda_n\,.
\]
This leads to a contradiction by the minimax characterization of the $n$th
eigenvalue.

Since we assume that $\lambda_n$ is Courant sharp, there exists an eigenfunction
$\Psi$ having exactly $n$ nodal components. Moreover, this $\Psi$ cannot be
radial (since $\ell\geq 1$). So we have $\Psi (r,\omega)= u_{\ell,m} (r)
\psi_\ell (\omega)$ where $\psi_\ell$ is a spherical harmonic. 
We let $\rho_1$ be the first zero of $u_{\ell,m}$ in $(0,1)$ (which exists 
since $m\geq 2$) and $\rho_2$ be the second zero,
if it exists, and $\rho_2=1$ otherwise. The ball is naturally divided
into the parts $|x|<\rho_1$ and $\rho_1<|x|<1$. Next, we define the
function $\widetilde\Psi$ as
\[
\widetilde\Psi(x)=
\begin{cases}
\Psi(\mathsf{R}x), & |x|<\rho_1\,,\\
\Psi(x), & \rho_1<|x|<1\,.
\end{cases}
\]
Here $\mathsf{R}$ is a small rotation, constructed in such a way that
the symmetry is broken.

Let us denote by $\widetilde{\mathcal D}=\bigcup_{j=1}^n\widetilde{D}_j$ 
the  twisted partition of nodal domains corresponding to $\widetilde\Psi$.

We now consider a pair of nodal domains of $\Psi$ in the form (after relabeling) 
  $D_{1}:= (0,\rho_1)\times \Omega$ and $D_{2}:=(\rho_1,\rho_2)\times\Omega$.
The twisting leads to the pair (see Figure~\ref{fig:twist}, middle subfigure) 
 $\widetilde D_{1}:=  (0,\rho_1)\times\mathsf R\Omega$ and   
$ \widetilde D_{2} = D_{2} :=   (\rho_1,\rho_2)\times\Omega$. Their boundary 
is $ \{\rho_1\}\times(\Omega \cap \mathsf R \Omega)$. 
The sets $\widetilde{D}_1$ and
$\widetilde{D}_2$ cannot be the $2$-partition of a second eigenfunction
in $\widetilde{D}_{1,2}:=\Int \Bigl( \overline{\widetilde D_1} 
\cup \overline{\widetilde D_{2}}\Bigr)$.  If it was true, it would exist 
$\mu \in \mathbb R $ 
such that $\mu \psi_\ell = \psi_\ell\circ \mathsf R$ in 
$\Omega\cap\mathsf R \Omega$. But 
this will imply  $\mu \psi_\ell = \psi_\ell\circ \mathsf R$ on 
$\mathbb S^{d-1}$ by analyticity.  We get a contradiction at the boundary 
of $\Omega$ or of $\mathsf R^{-1} \Omega$.

Thus, $\lambda_2(\widetilde{D}_{1,2})<\lambda_n$. By looking at the nodal set
of a second eigenfunction $u_{1,2}$ in $\widetilde{D}_{1,2}$, we get two
new sets $D_1'$ and $D_2'$ (the two nodal domains of $u_{1,2}$) such that
\[
\lambda_1(D_1')=\lambda_1(D_2')<\lambda_n\,.
\]
We recall that the remaining $n-2$ components of the partition 
$\widetilde{\mathcal D}$ have ground state energy $\lambda_n\,$. This is
illustrated in Figure~\ref{fig:twist}, to the right.
If $n=2$, then we are done. Below we assume that $n>2$.

We continue, by considering $D_1'$ or $D_2'$, and one of its neighbors, having
a boundary in common. Let us, for a while, denote this pair by  $D_1''$ and 
$D_2''$. It is possible, using the Hadamard formula (see~\cite{Henrot})  to
change the common boundary of $D_1''$ and $D_2''$ in such a way that two new
domains $\widetilde{D}_1''$ and $\widetilde{D}_2''$ are constructed, with
\[
\lambda_1(D_1'')
<\lambda_1(\widetilde{D}_1'')
\leq \lambda_1(\widetilde{D}_2'')
<\lambda_1(D_2'')\,.
\]
In particular,
\[
\max\Bigl(\lambda_1(\widetilde{D}_1''),\lambda_1(\widetilde{D}_2'')\Bigr)
<\max\Bigl(\lambda_1(D_1''),\lambda_1(D_2'')\Bigr)
=\lambda_n\,.
\]
At this point we have constructed three domains inside the ball, with 
ground state energy strictly less than $\lambda_n$. 
If $n=3$, we are done. If $n\geq 4$,
we continue this procedure recursively until all the
remaining domains in the partition $\widetilde{\mathcal{D}}$ have been
modified, and find in the end a new partition 
$\widehat{\mathcal D}$ of the ball, consisting of $n$ pairwise disjoint sets
$\widehat{D}_j$, such that $\lambda_1(\widehat{D}_j)<\lambda_n$ for all $j$.

The proof in the Neumann case is unchanged. One can do the necessary 
deformations in the boundaries where the Dirichlet condition is imposed.
\end{proof}

\begin{figure}
\centering
\makebox[\textwidth][c]{%
\includegraphics{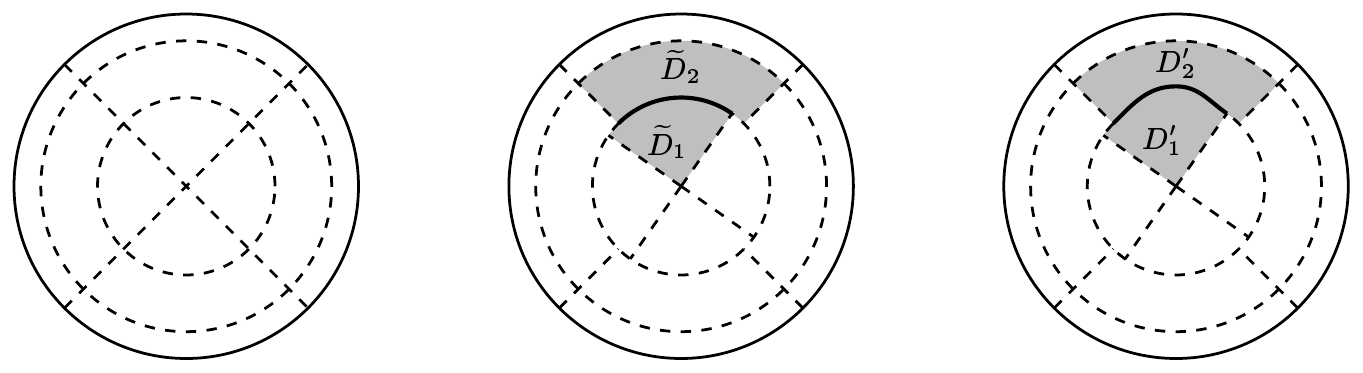}
}
\caption{Twist showcase. The inner dashed circle correspond to $|x|=\rho_1\,$, 
and the outer one to $|x|=\rho_2\,$.
Left: The partition of nodal domains of $\Psi$. 
Middle: The  nodal partition of $\widetilde{\Psi}\,$. The 
common boundary of $\widetilde{D}_1$ and $\widetilde{D_2}$ is marked thicker.
Right: The two nodal domains $D_1'$ and $D_2'$ of the second eigenfunction
in $\widetilde{D}_{1,2}\,$.
}
\label{fig:twist}
\end{figure}

\begin{remark}
The proof of Lemma~\ref{lem:twist} is easier in dimension two, since we
can refer to the \enquote{equal angle meeting} property, which is satisfied by 
any regular minimal partition  as established 
in~\cite{HHOT}.
\end{remark}

\subsection{Remaining eigenvalues, Dirichlet case}
\label{sec:Dirichlet}

\begin{lemma} Let $d\geq 3\,$. Then
\[
\lambda_{1,1}^D < \lambda_{0,2}^D\,.
\]
\end{lemma}

\begin{proof}
Denote by $j_{\nu,m}$ the $m$th positive zero of the Bessel function $J_\nu\,$.
The interlacing of zeros of $J_\nu$ (see~\cite[\S15.22]{Wa}),
\[
0<j_{\nu,1}<j_{\nu+1,1}<j_{\nu,2}<j_{\nu+1,2}<j_{\nu,3}
<\cdots\quad\forall \nu>-1\,,
\]
implies, with $\nu=d/2\,$, that $\alpha_{1,1}<\alpha_{0,2}\,$, and so 
$\lambda_{1,1}^D < \lambda_{0,2}^D\,$.
\end{proof}

Hence only $\lambda_{0,1}^D$ can be (and is!) Courant sharp in the list 
$\lambda_{0,m}^D$. For the sequence $\lambda_{\ell,1}^D\,$, one can use what we 
have proven for the sphere. Only $\lambda_{0,1}^D$ and $\lambda_{1,1}^D$ can 
be Courant sharp. This completes the proof of Theorem~\ref{thm:dD} 
for the Dirichlet problem.

\subsection{Remaining eigenvalues, Neumann case}
\label{sec:Neumann}
\begin{lemma} Let $d\geq 2\,$. Then
\[
\lambda_{1,1}^N < \lambda_{0,2}^N\,.
\]
\end{lemma}

\begin{proof}
We show that $\beta_{1,1}<\beta_{0,2}\,$. We recall that $\beta_{1,1}$ is the 
first \emph{positive} zero of $\Xi_1'$ and $\beta_{0,2}$ is the first 
\emph{positive} zero of $\Xi_0'$. But, according to~\eqref{eq:Xip1}, 
$\Xi_0'(r)=-\Xi_1(r)$. Now, $\Xi_1(r)= 2^{-d/2}/\Gamma(d/2+1)r+\mathcal O(r^3)$ 
as $r\to 0$, so, in particular $\Xi_1(0)=0$ for all $d\geq 2$. It follows
that $\beta_{1,1}<\beta_{0,2}$ by the mean value theorem.
\end{proof}

As a result, $\lambda_{0,m}^N$ cannot be Courant sharp if $m\geq 2\,$.
Indeed, since the eigenfunctions corresponding to $\lambda_{0,m}^N$ 
have precisely $m$ nodal domains, and the labeling of $\lambda_{0,m}^N$ is 
at least $m+1$ because $\lambda_{0,1}^N<\lambda_{1,1}^N<\lambda_{0,2}^N\,$.

We continue with the eigenvalues $\lambda_{\ell,1}^N\,$, and start
with the case $d=2$. We first see the following ordering 
for the eight first eigenvalues (see Figure~\ref{fig:nd}):
\[
\lambda_1^N=\lambda_{0,1}^N < \lambda_2^N = \lambda_{1,1}^N=\lambda_3^N < 
\lambda_4^N=\lambda_{2,1}^N =\lambda_5^N <\lambda_6^N=\lambda_{0,2}^N  < 
\lambda_7^N= \lambda_{3,1}^N =\lambda_8^N\,,\,\ldots
\]
with corresponding number of nodal domains
\[
\mu_1=1\,,\, \mu_2 =2\,,\,\mu_4 =4\,,\, \mu_6 =2\,,\, \mu_7 =6\,,\ldots
\]
We observe that $\lambda_{0,2}^N<\lambda_{3,1}^N\,$. Hence $\lambda_{\ell,1}^N$
cannot have a label lower than $2\ell+1$ in the complete ordered list of
eigenvalues and the corresponding eigenfunction has exactly $2\ell$ nodal
domains.

For $d\geq 3$ we can again use Theorem~\ref{thm:CourantSphere} to conclude 
that only the two first eigenvalues $\lambda_{0,1}^N$ and $\lambda_{1,1}^N$
can be Courant sharp.

This finishes the proof of Theorem~\ref{thm:dD} in the Neumann case.\qedhere

\begin{figure}
\centering
\makebox[\textwidth][c]{%
\includegraphics{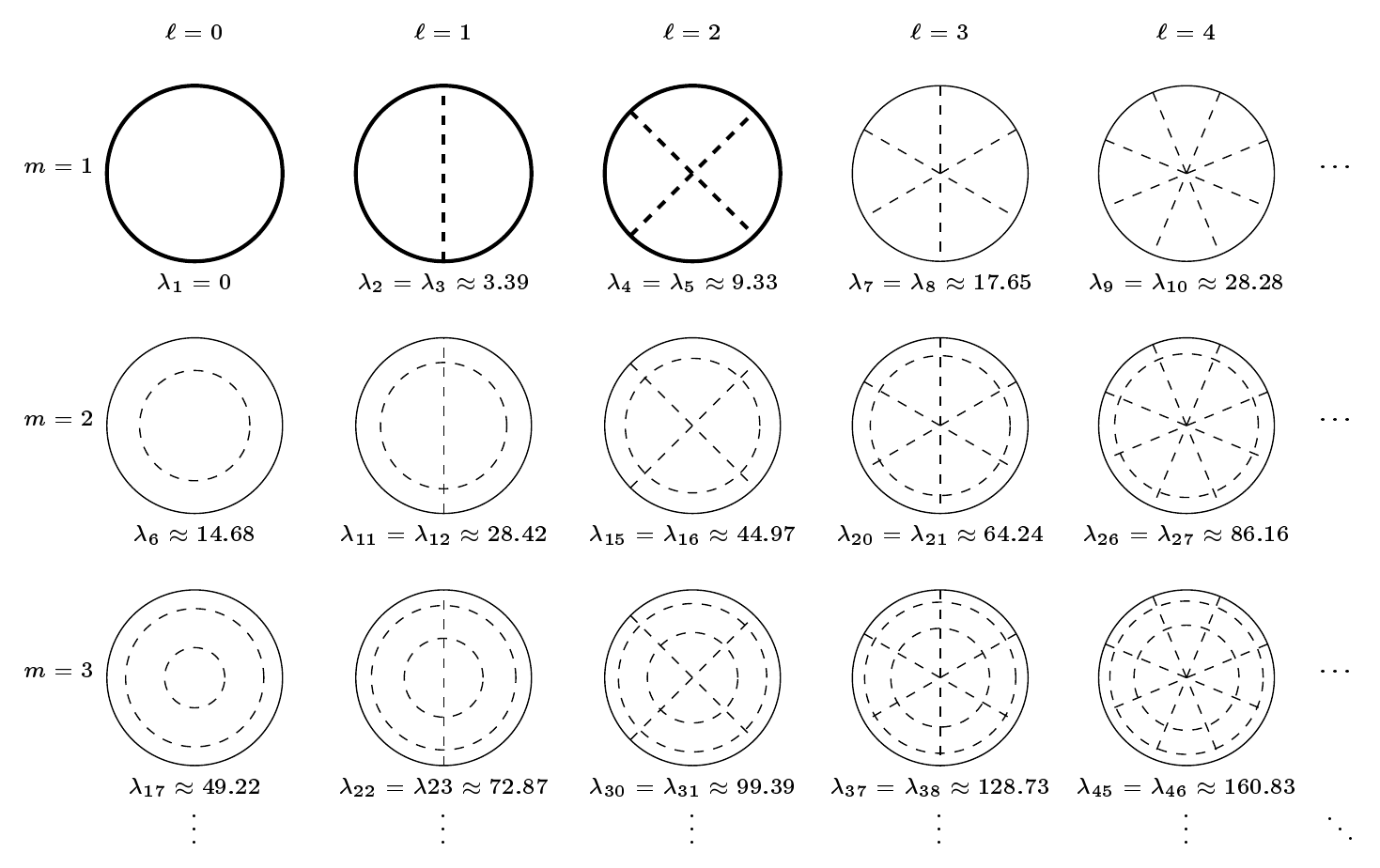}
}
\caption{Nodal domains for the eigenfunctions in the Neumann case for $d=2$. 
Thicker: The Courant sharp cases. The twisting argument excludes every case 
except $m=1$ (row one) and $\ell=0$ (column one) from being Courant sharp.}
\label{fig:nd}
\end{figure}

\section{On Pleijel's Theorem}
\label{sec:Pleijel}
We will discuss (the dimension-dependent) Pleijel constant
\[
\gamma (d):= C_{d}^{-1}  \, \bigl(\lambda_1(\mathbb{B}_1^d) \bigr)^{- d/2} < 1\,,
\]
where
$C_{d}$ is the Weyl constant $C_{d}:= (2\pi)^{-d}\omega_d$,
$\omega_d$ is the volume of the unit ball in $\mathbb R^d$,
and $\lambda_1(\mathbb{B}_1^d) $ is the Dirichlet ground state energy of
the
Laplacian in the ball  $\mathbb{B}_1^d$ of volume $1$.
More explicitly, we get (see~\cite[Lemma~9]{BeMe})
\[
\gamma (d)=
\frac{2^{d-2} d^2 \Gamma (d/2)^2}{(j_{\frac{d-2}{2},1})^d}\,.
\]
As explained in the introduction, we focus on the Neumann case.

\begin{thm}
For any infinite sequence of eigenpairs $(\lambda_n,u_n)$ of the Neumann 
Laplacian in the unit ball in $\mathbb R^d$ ($d\geq 3$),
\[
\lim\sup \frac{\mu(u_n)}{n} \leq  \gamma (d-1)  < 1\,.
\]
\end{thm}
 
We recall that $\gamma(2) = 4 / (j_{0,1}^2)$ and that 
$\gamma (3) = \frac{9}{2 \pi^2}$ (see~\cite{HK}). We refer the reader to the 
last section for further properties of $\gamma (d)$. The Neumann case is more 
delicate but can result for 
the disc of the general result of Polterovich for domains in $\mathbb R^2$ with 
piecewise analytic boundary. In~\cite{Pol} he shows that 
Pleijel's theorem holds with the same constant as for Dirichlet, as a 
consequence of a fine result due to Toth--Zelditch~\cite{ToZe} on the 
relatively small number of points at the 
intersection of the boundary and the zeroset. To our knowledge, nothing has been 
established in dimension $d\geq 3$ for the Neumann problem.

A natural idea is   to try to control the number of nodal domains touching the 
boundary on a set with non empty interior. This was the strategy proposed by 
Pleijel~\cite{Pl}  for the square and more generally by 
I.~Polterovich~\cite{Pol} for the $2D$-case.  We know indeed that it is 
$\mu(\psi_\ell)$ when the eigenfunction is $\Psi_{\ell,m}$. 
Hence the quotient between the number of \enquote{boundary} nodal sets 
divided by the total number tends to $0$, as $m \to +\infty$, like 
$\mathcal O(1/m)$. In this case, the \enquote{Faber--Krahn} proof 
works like in the Dirichlet case.

Hence it remains to control the case when $m \leq m_0$. In this case, $\ell$ 
tends to $+\infty$ as $\lambda^N_{\ell,m} \to + \infty$. 

We know that the labeling  of $\lambda_{\ell,m}^N$ is larger than 
$m\times n_\ell$ where $n_\ell$ is the labeling of $\ell (\ell +1)$. 
Hence we get 
\[
\lim\sup \frac{\mu(\Psi_{\ell,m})}{n_{\ell,m}} 
\leq \lim\sup \frac{\mu(\psi_\ell)}{n_\ell}
\leq \gamma (d-1)\,,
\]
where $n_{\ell,m}$ (resp. $n_\ell$) is the labeling of $\lambda_{\ell,m}^N$ 
for the Laplacian in the ball (resp. of $\ell (\ell+1)$ for the Laplacian on 
the sphere). For the last inequality, we have used B\'erard-Meyer (Pleijel like) 
theorem for the sphere $\mathbb S^{d-1}$.
At this stage, we have obtained 
\[
\lim\sup \frac{\mu(u_n)}{n} \leq  \max\bigl( \gamma (d), \gamma (d-1) \bigr) < 1\,.
\]
The conclusion of the theorem is obtained using the monotonicity of $\gamma$ 
which will be established in the next section.


\section{Monotonicity of $\gamma(d)$}
\label{sec:gamma}
 
We recall from \cite[Lemma~9]{BeMe} that $\gamma(d)<1$. 
The proof of this inequality relies on the estimate
\[
\Gamma (x) \leq x^{x-\frac 12} \, e^{-x} \, \sqrt{2\pi} \, e^{\frac {1}{12x}}\,,
\]
which gives the estimate for $d \geq 18$ and of numerical computations 
(see below)  for $d <18$. 
It is natural to discuss about the monotonicity of $d\mapsto \gamma(d)$.

\begin{table}\label{tb1}
\caption{Table of values of $\displaystyle \gamma(d) $}
\begin{tabular}{crrrrr}
\toprule
$d$ & 2 & 3 & 4 & 5 & 6 \\
$\gamma(d)$ & 0.691660 & 0.455945 & 0.296901 & 0.192940 & 0.125581 \\
\midrule
$d$ & 7 & 8 & 9 & 10 & 11 \\
$\gamma(d)$ & 0.081982 & 0.053704 & 0.035306 & 0.023291 & 0.015417 \\
\midrule
$d$ & 12 & 13 & 14 & 15 & 16 \\
$\gamma(d)$ & 0.010236 & 0.006817 & 0.004553 & 0.003048 & 0.002046 \\
\midrule
$d$ & 17 & 18 & 19 & 20 & 21 \\
$\gamma(d)$ & 0.001376 & 0.000928 & 0.000627 & 0.000424 & 0.000288 \\
\bottomrule
\end{tabular}
\end{table}

\begin{thm}
\label{thm:gammamono}
The function $d\mapsto \gamma(d)$ is monotonically decreasing.
\end{thm}
\begin{proof}
We first recall  that
\begin{lemma}\label{lemmatech1}
$x \mapsto \Gamma(x)$ is logarithmically convex.
\end{lemma}
We write
\[
\Gamma \Bigl(\frac d2 +\frac 12\Bigr)/\Gamma \Bigl(\frac d2\Bigr) 
=\frac 12\,  (d-1) \, \Gamma \Bigl( \frac d2 -\frac 12\Bigr)/\Gamma \Bigl(\frac d2\Bigr)\,.
\]
and use the logarithmic convexity of the $\Gamma$-function,
\[
\Gamma \Bigl( \frac d2 -\frac 12\Bigr)
\leq \sqrt{\Gamma \Bigl(\frac d2 -1\Bigr)} \sqrt{\Gamma \Bigl(\frac d2\Bigr)}\,.
\]
This implies
\[
 \Gamma \Bigl( \frac d2 -\frac 12\Bigr)/\Gamma \Bigl(\frac d2\Bigr) 
 \leq \sqrt{\Gamma \Bigl(\frac d2 -1\Bigr) 
/ \Gamma \Bigl(\frac d2\Bigr)} =\frac{1}{\sqrt{\frac d2 -1}}\,.
\]
Finally, we deduce:
\begin{equation}
\label{eq:Gammaeq}
\biggl[\Gamma \Bigl(\frac d2 +\frac 12\Bigr)
/\Gamma \Bigl(\frac d2\Bigr)\biggr]^2 \leq \frac {(d-1)^2}{2(d -2)}\,.
\end{equation}

\paragraph{We next estimate $j_{d/2 -1,1}/ j_{d/2 -1/2,1}$.}

A less known result by Lewis--Muldoon is
\begin{lemma}[{\cite[Formula~(1.2)]{LeMu}}]
\label{lemmatech2}
For $\nu \geq 3$, $\nu \mapsto j_{\nu,1}^2$ is convex.
\end{lemma}

We write the convexity of $j_{\nu,1}^2$ (for $\nu \geq 3$)
\[
 j_{\nu-1,1}^2 \leq \frac 12 \bigl(  j_{\nu-3/2,1}^2 + j_{\nu-1/2,1}^2\bigr)\,.
\]
This gives
\begin{equation}
\label{eq:jinter}
\frac{j_{\nu-1,1}^2 }{ j_{\nu-1/2,1}^2} 
\leq \frac 12 \biggl(  \frac{j_{\nu-3/2,1}^2}{j_{\nu-1/2,1}^2}+1\biggr)\,.
\end{equation}

In \cite{AB},   Ashbaugh and Benguria prove:
\[
 \frac{d}{d+4} < \frac{(j_{d/2-1,1})^2}{(j_{d/2,1})^2}\leq \frac{\lambda_1(\Omega)}{\lambda_2(\Omega)},
\]
where $\Omega$ is a $d$-dimensional domain, the right hand side inequality 
being attained for the ball and the left hand side being the Thomson 
inequality. In particular, for the $d$-cube,
$\lambda_1(\Omega)=d$ and $\lambda_2(\Omega)=d+3$. This implies that
\begin{equation} \label{eq:ASB}
\frac{j_{d/2-1,1}}{j_{d/2,1}}
\leq
\Bigl(1-\frac{3}{d+3}\Bigr)^{1/2}\,. 
\end{equation}
To estimate the right hand side of~\eqref{eq:jinter}, we use~\eqref{eq:ASB} 
(with $d$ replaced by $d-1$) to obtain
\begin{equation} \label{eq:control}
\frac{ j_{d/2-1,1}^2}{j_{d/2-1/2,1}^2} 
\leq 
\frac 12 \biggl(  \frac{j_{(d-1)/2-1,1}^2}{j_{(d-1)/2,1}^2}+1\biggr)
\leq  1 - \frac{3}{2(d+2)}\,.
\end{equation}
To our knowledge the best estimates for the zeros of Bessel functions are
\begin{equation}
\label{eq:Jest}
\sqrt{\nu(\nu+2)}<j_{\nu,1}<\sqrt{\nu+1}\, \bigl(\sqrt{\nu+2}+1\bigr)\,.
\end{equation}
The left estimate is available in Watson~\cite{Wa}, the right one was proven by
Chambers~\cite{Ch}, by choosing a good trial state for the Rayleigh quotient.

Inserting~\eqref{eq:Gammaeq},~\eqref{eq:control} and the lower bound 
in~\eqref{eq:Jest} into the 
quotient $\gamma(d+1)/\gamma(d)$, we get the following bound,
\[
\frac{\gamma (d+1)}{\gamma (d)}
\leq 2 \frac{1}{\sqrt{ (d-1)(d+3)} }\,  \frac{(d+1)^2}{d^2}\, 
\frac { (d-1)^2}{(d -2)} \Bigl(  1 - \frac{3}{2(d+2)}\Bigr)^{d/2}\,.
\]
Next, we use the inequality $(1+a/x)^{x}\leq e^a$,
\[
\Bigl(1-\frac{3}{2(d+2)}\Bigr)^{d/2+1}=\Bigl(1-\frac{3/4}{(d+2)/2}\Bigr)^{(d+2)/2}
\leq e^{-3/4}\,,
\]
to get
\[
\frac{\gamma (d+1)}{\gamma (d)}
\leq
\frac{2}{e^{3/4}}
\frac{1}{\sqrt{ (d-1)(d+3)} }\,  
\frac{(d+1)^2}{d^2}\, \frac { (d-1)^2}{(d -2)} \Bigl(1 - \frac{3}{2(d+2)}\Bigr)^{-1}\,.
\]
We estimate $(d-1)(d+3)\geq (d+1/2)^2$ (valid for $d\geq 4$), and
write $1-3/(2(d+2))$ as $(d+1/2)/(d+2)\,$, to find
\[
\frac{\gamma (d+1)}{\gamma (d)}
\leq
\frac{2}{e^{3/4}}\,  \frac{(d+1)^2}{d^2}\, \frac { (d-1)^2}{(d -2)} \frac{(d+2)}{(d+1/2)^2}\, .
\]
We next show that the right-hand side is bounded by $2/e^{3/4}(1+5/d)$ if
$d\geq 4\,$. For this purpose, we write 
\[
\begin{multlined}
\frac{(d+1)^2}{d^2}\, \frac { (d-1)^2}{(d -2)} \frac{(d+2)}{(d+1/2)^2}
-(1+5/d)\\
=\frac{d^5(-4-12/d+39/d^2+41/d^3-2/d^4-8/d^5)}{4(d-2)d^2(d+1/2)^2}\,.
\end{multlined}
\]
Since 
\[
-4-12/d+39/d^2+41/d^3-2/d^4-8/d^5\leq -4+39/d^2+41/d^3\,,\quad\forall d\geq 1\,,
\]
and $-4+39/d^2+41/d^3$ is monotonically decreasing
and equals $-59/64$ for $d=4\,$, we find that
\[
\frac{(d+1)^2}{d^2}\, \frac { (d-1)^2}{(d -2)} \frac{(d+2)}{(d+1/2)^2}
\leq (1+5/d),\quad\forall d\geq 4\,.
\]
Thus
\[
\frac{\gamma (d+1)}{\gamma (d)}
\leq
\frac{2}{e^{3/4}}(1+5/d),\quad\forall d\geq 4\,.
\]
The numerical approximation $2/e^{3/4}\approx 0.945$ implies that 
$2/e^{3/4}<95/100\,$, and so
\[
\frac{\gamma (d+1)}{\gamma (d)}
<
\frac{95}{100}\Bigl(1+\frac{5}{d}\Bigr),\quad\forall d\geq 4\,.
\]
The right-hand side is monotonically decreasing in $d$, and equals one for $d=95\,$. Hence,
\[
\frac{\gamma (d+1)}{\gamma (d)}<1,\quad \forall d\geq 95\,.
\]
The remaining cases, $2\leq d\leq 94\,$, are covered numerically using
Mathematica, and the quotient is plotted in Figure~\ref{fig:gammaquotient}.
This finishes the proof of Theorem~\ref{thm:gammamono}.
\end{proof}

\begin{figure}\centering
\includegraphics{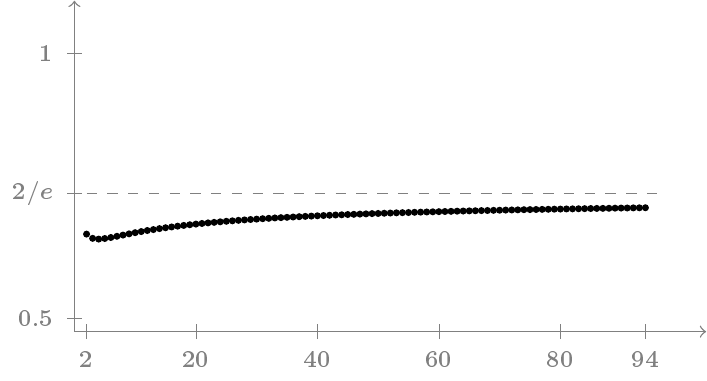}
\caption{The quotient $\gamma(d+1)/\gamma(d)$, for $2\leq d\leq 94\,$. The 
limit is $2/e\,$.}
\label{fig:gammaquotient}
\end{figure}

\begin{remark} As shown in Figure \ref{fig:gammaquotient}, a classical 
asymptotics for $j_{\nu,1}$ gives, after observing that 
\[
\lim_{d\to +\infty} \frac{\gamma (d+1)}{\gamma (d)} 
=\sqrt{ \lim_{d\to +\infty} \frac{\gamma (d+2)}{\gamma (d)} }
\]
the following limiting value for the quotient $\gamma (d+1)/\gamma(d)$:
\[
\lim_{d\to + \infty} \frac{\gamma (d+1)}{\gamma (d)} = \frac{2}{e}\,.
\]
\end{remark}


\section*{Acknowledgements}
We thank I.~Polterovich, E.~Wahl\'{e}n and J. Gustavsson for stimulating 
discussions and informations.

\end{document}